\documentclass[11pt,a4paper]{amsart}

\usepackage{latexsym}
\usepackage[colorlinks]{hyperref}
\usepackage{xcolor}
\usepackage{mathscinet}
\usepackage{amsthm}
\usepackage{amssymb}
\usepackage{amsfonts}
\usepackage{amsmath}
\title[Ergodicity of $\alpha$-stable processes]
{Ergodicity of infinite white $\alpha$-stable Systems with linear and bounded interactions}
\author[L.Xu]{Lihu Xu}
\address{PO Box 513, EURANDOM, 5600 MB  Eindhoven. The Netherlands}
\email{xu@eurandom.tue.nl}
\date{}
 \newtheorem{thm}{Theorem}[section]
 
 \newtheorem{lem}[thm]{Lemma}
 \newtheorem{prop}[thm]{Proposition}
 \newtheorem{assumption}[thm]{Assumption}
 \theoremstyle{definition}
 
 \theoremstyle{remark}
 \newtheorem{rem}[thm]{Remark}
 \numberwithin{equation}{section}

 \newcommand{\R}{\mathbb{R}}

 \newcommand{\G}{\mathcal{L}}

 \newcommand{\p}{\partial}
 \newcommand{\D}{\mathcal{D}}
 \newcommand{\e}{\varepsilon}
 \newcommand{\Z}{\mathbb{Z}}
 \newcommand{\E}{\mathbb{E}}
 \newcommand{\N}{\mathbb{N}}
 \newcommand{\mcl}{\mathcal}

\begin{document}
\maketitle
%
\begin{abstract} \label{abstract}
We proved the existence of an infinite dimensional stochastic system driven by
white $\alpha$-stable noises ($1<\alpha \leq 2$), and prove this system is strongly mixing.
Our method is by perturbing Ornstein-Uhlenbeck $\alpha$-stable processes.
 \\

\end{abstract}
\section{Introduction}\label{Introduction}
We shall study an infinite dimensional spin system with
linear and bounded interactions, driven by (white) $\alpha$-stable
noises. More precisely, our system is described by the following
SDEs: for every $i \in \mathbb{Z}^d$,
\begin{equation} \label{e:IntSys}
\begin{cases}
dX_i(t)=[\sum_{j \in \Z^d} a_{ij} X_j(t)+U_i(X(t))]dt+dZ_i(t) \\
X_i(0)=x_i
\end{cases}
\end{equation}
where $X_i, x_i \in \R$, $\{Z_i; i \in \Z^d\}$ are a sequence of i.i.d.
standard symmetric $\alpha$-stable
process with $1<\alpha \leq 2$, and the assumptions for the $a$ and $U$ are specified in Assumption
\ref{a:AU}.
\ \\

When $Z(t)$ is Wiener noise, the equation \eqref{e:IntSys} has been intensively studied
to model some phenomena in physics such as quantum spin systems
since the 90s of last century
(\cite{AKYR95}, \cite{AKYT94}, \cite{DPZ96}, \cite{PeZa07},
\cite{AMR09},\cite{PeZa06}, \cite{PrZa09}, \cite{PrZa09-2}, $\cdots$).
The other motivation to study \eqref{e:IntSys} is from the work by Zegarlinski
on interacting unbounded spin systems driven by Wiener noise (\cite{Ze96}).
In that paper, the author proved the following uniform ergodicity
$||P_tf-\mu(f)||_{\infty} \leq C(f) e^{-m t}$, where
$P_t$ is the semigroup generated by a reversible generator and $\mu$
is the ergodic measure of $P_t$. This type of ergodicity is
very strong, and obtained by a logarithmic Sobolev inequality (LSI). Unfortunately,
the LSI is not available in our set-up, however, we can find a \emph{gradient} decay estimates,
which is crucial in the proof of ergodic theorem. We should point out that our ergodicity result
is strongly mixing type, which is much weaker than the uniform one.
\ \\

The approach of this paper is via perturbing the Ornstein-Uhlenbeck
$\alpha$-stable process, this needs one to know some exact formula
of this process. Comparing with the above perturbation approach,
the main tools in \cite{XZ09-1}
are some iterations under the framework of probability, one
only uses the $\alpha$-stable property
and the moments of the stable processes.
Hence, we can think that \cite{XZ09-1} is some \emph{qualitative} analysis,
while this paper is some \emph{quantitative} one.
\\

The organization of the paper is as follows. The introduction includes notations, main
results and some
preliminary about Ornstein-Uhlenbeck $\alpha$-stable processes, the second and third sections
prove the existence and ergodicity results respectively. The short appendix gives a simple but
interesting derivation of \eqref{e:ProRepU}.
\subsection{Notations, Assumptions and Main Results}
We shall study the system \eqref{e:IntSys} on $\mathbb{B} \subset \mathbb{R}^{\mathbb{Z}^d}$,
which is defined by $$\mathbb{B}=\bigcup \limits_{R>0, \rho>0}B_{R,\rho}$$
where for any $R,\rho>0$
 $$\ B_{R,\rho}=\{x=(x_i)_{i \in \Z^d};
 |x_i| \leq R (|i|^{\rho}+1)\} \ \ {\rm with} \ \ |i|=\sum \limits_{k=1}^d |i_k|.$$
We shall see that given any initial data $x=(x_i)_{i \in \Z^d} \in \mathbb B$,
the dynamics $X(t)$ defined
in \eqref{e:IntSys} evolves in $\mathbb B$ almost surely.
\begin{rem} \label{r:RemB}
One can also check that the distributions of the standard white $\alpha$-stable
processes $(Z_i(t))_{i \in \Z^d}$ ($0<\alpha \leq 2$) at any fixed time $t$
are supported on $\mathbb B$.
From the form of the equation \eqref{e:IntSys}, one can expect that the distributions
of the system at any fixed time $t$ is similar to those of $\alpha$-stable processes
but with some (complicated) shifts. Hence, it is \emph{natural} to study \eqref{e:IntSys} on $\mathbb B$.
\end{rem}

Let us first list some notations which will be frequently used in the paper
and then give the detailed assumption on $a$ and $U$.
\begin{itemize}
\item Define $|i-j|=\sum
_{1 \leq k \leq d}|i_k-j_k|$ for any $i,j \in \Z^d$,
define $|\Lambda|=\sharp \Lambda$ for any finite sublattice $\Lambda \subset \subset \Z^d$.
\item For the national simplicity, we shall write
$\partial_i:=\partial_{x_i}$, $\p_{ij}:=\p^2_{x_i x_j}$
and $\partial^{\alpha}_i:=\partial^{\alpha}_{x_i}$. It is easy to see that
$[\p^{\alpha}_i, \p_j]=0$ for all $i,j \in \Z^d$.
\item For any finite sublattice $\Lambda \subset \subset \mathbb{Z}^d$,
let $C_b(\R^{\Lambda},\R)$ be the bounded continuous
function space from $\R^{\Lambda}$ to $\R$, denote
$\mathcal{D}=\bigcup_{\Lambda \subset \subset
\mathbb{Z}^d} C_b(\R^{\Lambda},\R)$ and
$$\mathcal{D}^{k}=\{f \in
\mathcal{D}; f \ {\rm has} \  {\rm bounded} \ 0, \cdots,\ kth \ {\rm order} \ {\rm derivatives}\}.$$
\item For any $f \in \mathcal{D}$, denote $\Lambda(f)$ the
localization set of $f$, i.e. $\Lambda(f)$ is the smallest set
$\Lambda \subset \mathbb{Z}^d$ such that $f \in C_b(\R^{\Lambda},\R)$.
\item For any $f \in C_b(\mathbb B,\R)$, define $||f||=\sup_{x \in \mathbb B} |f(x)|$.
For any $f \in \mcl D^1$, define $|\nabla f(x)|^2=\sum_{i \in \Z^d} |\p_i f(x)|^2$.
\item For any $f \in \mcl D^1$, define $|\nabla f(x)|^2=\sum_{i \in \Z^d} |\p_i f(x)|^2$.
\item $||\cdot||$ is the uniform norm, i.e. for any $f \in C_b(\mathbb B, \R)$,
$||f||=\sup_{x \in \mathbb B} |f(x)$. The seminorms $|||\cdot|||_1$ and $|||\cdot|||_2$
are respectively defined by
$$|||f|||_1=\sum \limits_{i \in \mathbb{Z}^d}||\partial_i f|| \ \ \ \ \ \ \ f \in \mathcal{D}^{1},$$
and
$$|||f|||_2=\sum \limits_{j,k \in \mathbb{Z}^d}||\partial_{jk} f|| \ \ \ \ \ \ \
f \in \mathcal{D}^{2}.$$
\item $\mcl B_b(H,\R)$ is the function space including the bounded measurable functions from some
topological space $H$ to $\R$,
\end{itemize}

\begin{assumption} [Assumptions of $a$ and $U$] \label{a:AU}
The $a$ and $U$ in \eqref{e:IntSys} satisfies the following conditions:
\begin{enumerate}
\item \emph{(Linear interactions)} $a_{ij} \geq 0$ for all $i \neq j$, $a_{ii}=-1$ for all $i \in \Z^d$.
\item \emph{(Bounded interactions)} $U_i \in \mcl D^2$ for all $i \in \Z^d$,
$\sup\limits_{i \in \Z^d} ||U_i||<\infty$.
\item  \emph{(Finite range property)} There exists some $K \in \N$ such that,
for all $i,j  \in \Z^d$ with $|i-j|>K$, one has $a_{ij}=0$ and $\p_j U_i(x)=0$ for all $x \in \mathbb B$.
\item $\eta<\infty$ with $\eta:=\sup \limits_{j \in \Z^d}
\left(\sum \limits_{i \in \Z^d, i \neq j} a_{ij}+|||U_j|||_1\right)$,
and $\sup \limits_{j \in \Z^d} |||U_j|||_2<\infty$.
\end{enumerate}
\end{assumption}
\ \\

The main results of this paper are the following two theorems
\begin{thm} \label{t:ConDyn}
There exists a Markov semigroup $P_t$ on the space $\mcl B_b(\mathbb B, \R)$
generated by the system \eqref{e:IntSys}.
\end{thm}
\begin{thm} \label{t:ErgThm}
We have some constant $c>0$ such that as $\eta<c$,
there exists a probability measure $\mu$ supported on $\mathbb B$
so that for all $x \in \mathbb B$,
$$\lim_{t \rightarrow \infty} P^*_t \delta_x = \mu \ \ \ {\rm weakly}.$$
\end{thm}
\begin{rem}
The convergence in Theorem \ref{t:ErgThm} implies that $P_t$ is strongly mixing (see \cite{DPZ96}).
\end{rem}
\subsection{Ornstein-Uhlenbeck $\alpha$-stable Processes}
\label{Section2}
\subsubsection{One dimensional Ornstein-Uhlenbeck $\alpha$-stable process}
This process is described by the
following stochastic differential equation (SDE)
\begin{equation} \label{e:OUa}
\begin{cases}
dX(t)=-X(t)dt+dZ(t) \\
X(0)=x
\end{cases}
\end{equation}
where $Z(t)$ is a symmetric $\alpha$-stable process ($0<\alpha \leq 2$) with
infinitesimal generator $\partial^{\alpha}_x$ defined by
\begin{equation} \label{e:fraclap}
\partial^{\alpha}_x f(x)=\frac{1}{C_{\alpha}} \int_{\mathbb{R} \setminus \{0\}} \frac{f(y+x)-f(x)}{|y|^{\alpha+1}}dy, \ \ C_{\alpha}=
-\int_{\mathbb{R} \setminus \{0\}} (cosy-1)\frac{dy}{|y|^{1+\alpha}}.
\end{equation}
as $0<\alpha <2$, and by $\frac{1}{2}\Delta$ as $\alpha=2$ (\cite{ARW00}). Moreover, if $f$ has
Fourier transform $\hat{f}$, then
$$\partial^{\alpha}_x f(x)=\frac{1}{\sqrt{2 \pi}} \int_{\mathbb{R}} |\lambda|^{\alpha}
\hat{f}(\lambda)e^{i \lambda x}d \lambda.$$
\\

The Kolmogorov backward equation of (\ref{e:OUa}) is
\begin{equation} \label{e:kola}
\begin{cases}
\partial_t u=\partial^{\alpha}_x u-x\partial_x u, \\
u(0)=f,
\end{cases}
\end{equation}
which is solved by
\begin{equation} \label{e:ProRepU}
u(t,x)=\E_x[f(X(t))]=\int_{\mathbb R} p\left(\frac{1-e^{-\alpha t}}{\alpha};e^{-t}x,y\right) f(y)dy,
\end{equation}
where $X(t)$ is the solution to \eqref{e:OUa} and
\begin{equation} \label{e:AlpStaTra}
p(t;x,y)=\frac{1}{\sqrt{2 \pi}}\int_{{\mathbb R}} \frac{1}{\sqrt{2
\pi}} e^{-t|\lambda|^{\alpha}+i(x-y) \lambda} d\lambda.
\end{equation}
One can refer to the Appendix for a formal derivation of \eqref{e:ProRepU} and
\eqref{e:AlpStaTra}, or refer to \cite{XZ09} for the rigorous one. From \eqref{e:ProRepU}, one can easily see that as $t \rightarrow \infty$
\begin{equation} \label{e:Erg1d}
u(t,x) \rightarrow \int_{\R} p(1/\alpha;0,y)f(y)dy.
\end{equation}
Hence, the law of $X(t,x)$ weakly convergence to the measure $p(1/\alpha;0,y)dy$, which is
independent of the initial data $x$. Hence, $X(t,x)$ is ergodic and strongly mixing (\cite{DPZ96}).
\ \\

Define $S^0_t f(x):=\E_x[f(X_t)],$
which is the Ornstein-Uhlenbeck $\alpha$-stable semigroup generated by
the operator $\partial^{\alpha}_x-x\partial_x$.
\begin{lem} \label{l:StXBeta}
Let $1 \leq \beta<\alpha$, then
\begin{equation} \label{e:PolEst}
S^0_t[|x|^{\beta}](0) \leq C(\beta).
\end{equation}
\end{lem}
\begin{proof}
Since the symmetric $\alpha$-stable process ($1<\alpha \leq 2$) has $\beta$ order
moments with $1 \leq \beta<\alpha$, by \eqref{e:ProRepU} and \eqref{e:AlpStaTra}, we have
\begin{equation*}
\begin{split}
S^0_t[|x|^\beta](0)&=\int_{\R}
p\left(\frac{1-e^{-\alpha t}}{\alpha};0, y\right)
|y|^{\beta}dy\\
&=\left [\frac{1-e^{-\alpha t}}{\alpha}
\right]^{\beta/\alpha} \int_\R p(1;0,y)|y|^\beta dy \leq C(\beta)
\end{split}
\end{equation*}
\end{proof}
\subsubsection{Infinite dimensional Ornstein-Uhlenbeck $\alpha$-stable processes}
Consider the white symmetric $\alpha$-stable
noises $(Z_i(t))_{i \in \Z^d}$, i.e. $(Z_i(t))_{i \in \Z^d}$
are i.i.d. symmetric $\alpha$-stable processes, and define the infinite
dimensional Ornstein-Uhlenbeck $\alpha$-stable processes by the following SDEs: for all $i \in \Z^d$
$$dX_i(t)=-X_i(t)dt+dZ_i(t).$$
Clearly, $X_i(t)$ at each $i \in Z^d$ is an Ornstein-Uhlenbeck $\alpha$-stable process,
which is independent of the processes on the other sites.
By \eqref{e:Erg1d},  $(X_i(t))_{i \in Z^d}$ is ergodic and has
a unique invariant measure $\left(p(1/\alpha;0,y_i)dy_i\right)_{i \in \Z^d}$
with $p$ defined by \eqref{e:AlpStaTra}.\\

Moreover, it is easy to see that the infinitesimal generator of $(X_i(t))_{i \in \Z^d}$ is given by
\begin{equation} \label{e:OrnUhlPro}
\mcl L:=\sum_{i \in \Z^d} (\partial_i^{\alpha}-x_i\partial_i),
\end{equation}
which is well defined on $\mathcal D^\infty$. Clearly, $\mcl L$ generates
a Markov semigroup $S_t$ on $\mcl D$, which is the product of one dimensional Ornstein-Uhlenbeck $\alpha$-stable semigroup, more precisely,
for any function $f \in \mcl D$,
$S_tf$ is defined by
\begin{equation} \label{e:DefSt}
S_t f(x)=\int_{\R^{\Lambda(f)}} \prod_{j \in \Lambda(f)} p\left(\frac{1-e^{-\alpha t}}{\alpha};e^{-t}x_j,y_j\right) f(y) \prod_{j \in \Lambda(f)} dy_j
\end{equation}
\ \\
\section{Existence of Infinite Dimensional Interacting $\alpha$-stable Systems}

\subsection{Galerkin approximation of the interacting systems}
Let $\Gamma_N=[-N,N]^d$ be the cube of ${\mathbb Z}^d$.
We approximate the infinite dimensional system by
\begin{equation} \label{e:GalApp}
\begin{cases}
d X^N_i(t)=[\sum_{j \in \Gamma_N} a_{ij}
X^N_j(t)+U^N_{i}(X^N(t))]dt+dZ_i(t), \\
X^N_i(0)=x_i,
\end{cases}
\end{equation}
for all $i \in \Gamma_N$,
where $U^N_{i}(x^N)=U_i(x^N,0)$ with $x^N=(x_i)_{i \in \Gamma_N}$.
Since the coefficients in \eqref{e:GalApp} are Lipschitz, by \cite{Bi02} (chapter 5),
\eqref{e:GalApp} has a unique global strong solution. Moreover, for any
differentiable $f$, $P^N_tf$ is also differentiable (c.f. chapter 5 of \cite{Bi02}).
\\

For any $f \in \mcl D^\infty$ with
$\Lambda(f) \subset \Gamma_N$, define
$$P^N_tf(x)=E_x^N [f(X^N(t))],$$
then $P^N_t$ is a Markov semigroup, moreover, it satisfies
\begin{equation} \label{e:kolmogorov}
\begin{cases}
\partial_t u(t)=\mathcal{L}_N u(t)\\
u(0)=f
\end{cases}
\end{equation}
where $\mathcal{L}_N$ is the generator of $P^N_t$ defined by
\begin{align}
\mathcal{L}_N&=\sum_{i \in \Gamma_N} \p^{\alpha}_i+
\sum_{i \in \Gamma_N}[\sum_{j \in \Gamma_N} a_{ij}x_j+U^N_{i}(x)] \p_i \\
&=\sum_{i \in \Gamma_N} [\p^{\alpha}_i-x_i \p_i]+
\sum_{i \in \Gamma_N}[\sum_{j \in \Gamma_N \setminus i } a_{ij}x_j+U^N_{i}(x)] \p_i. \label{e:LNOrnUhl}
\end{align}
\ \\
From the form \eqref{e:LNOrnUhl} of $\mathcal{L}_N$, by Du Hamel principle, we have
\begin{equation} \label{e:DuHam}
P^N_tf=S_t f+\int_0^t S_{t-s} \left \{\sum_{i \in \Gamma_N}[\sum_{j \in \Gamma_N \setminus i } a_{ij}x_j+U^N_{i}(x)] \p_i P^N_s f \right\}ds
\end{equation}
where $S_t$ is defined by \eqref{e:DefSt}.

\subsection{Auxiliary Lemmas}
The following relation \eqref{e:expfinite} is usually called finite
speed propagation of interactions (\cite{GuZe03}),
which roughly means that the effects of the initial condition (i.e.
$f$ in our case) need a long time
to be propagated (by interactions) far away. The main reason for this phenomenon is that
the interactions are finite range.
\begin{lem}  \label{l:FinSpePro}
\ \\
\noindent 1. For any $f \in \D^\infty$, we have
\begin{equation} \label{e:TriBar}
|||P^N_tf|||_1 \leq e^{(\eta+1)t} |||f|||_1.
\end{equation}
\noindent 2. \emph{(Finite speed of propagation of the interactions)}
Given any $f \in \D^{1}$ and $k \notin \Lambda(f)$, for any $A>0$, there exists
some $B \geq 1$ such that when $n_k>Bt$, we have
\begin{equation} \label{e:expfinite}
||\partial_k P^N_t f||\leq e^{-At-An_k} |||f|||_1
\end{equation}
where $n_k=[\frac{dist(k,\Lambda(f))}{K}]$ and $K \in \N$ is defined in Assumption \ref{a:AU}.
\end{lem}

\begin{proof}
For the notational simplicty, we drop the index $N$ of the quantities if no confusions arise. By Markov property of $P_t$ and the easy fact
$$\frac{d}{ds}P_{t-s}\partial_k P_s
f=P_{t-s}[\partial_k, \mathcal{L}_N] P_s f$$
where
$[\partial_k,\mathcal{L}_N]=\p_k \mathcal{L}_N-\G_N
\partial_k=\sum_{i \in \Gamma_N} [a_{ik}+\p_k U_i]\p_i$, we have
\begin{equation} \label{e:PNtForIte}
\begin{split}
||\partial_k P_tf|| \leq ||\partial_k f||+\int_{0}^{t} \sum_{i \in
\Gamma_N} (|a_{ik}|+||\partial_k U_i||) ||\partial_i P_s
f||ds.
\end{split}
\end{equation}
To prove \eqref{e:TriBar}, summing the index $k$ of the above inequality over $\Z^d$,
by Assumption \ref{a:AU}, we have
\begin{equation*}
|||P_tf|||_1 \leq |||f|||_1+(1+\eta) \int_0^t |||P_sf|||_1ds,
\end{equation*}
which immediately implies \eqref{e:TriBar}.
\\

Now let us show \eqref{e:expfinite}. Denote $c_{ik}=||\partial_k U_{i}||$ and $\delta_{ik}$
Krockner's function, we have
by iterating \eqref{e:PNtForIte}
\begin{equation*}
\begin{split}
& ||\partial_k P_tf|| \leq \sum_{n=0}^{\infty} \frac{t^n}{n!} \sum_{i \in \Lambda(f)} [(c+a+\delta)^n]_{ik} ||\partial_i f|| \\
&=\sum_{n=0}^{n_k} \frac{t^n}{n!} \sum_{i \in \Lambda(f)} [(c+a+\delta)^n]_{ik}
||\partial_i f||+\sum_{n=n_k+1}^{\infty} \frac{t^n}{n!} \sum_{i \in \Lambda(f)} [(c+a+\delta)^n]_{ik} ||\partial_i f||
\end{split}
\end{equation*}
The first term of the last line is zero,
since $(\delta+a+c)_{ik}=0$ for all $n \leq n_k$ by the definition of $n_k$.
As for the second term, we can easily have
$$\sum_{n=n_k+1}^{\infty} \frac{t^n}{n!} \sum_{i \in \Lambda(f)}
[(c+a+\delta)^n]_{ik} ||\partial_i f|| \leq \frac{t^{n_k}}{n_k!} e^{(1+\eta)t} |||f|||_1.$$
Hence,
\begin{equation} \label{e:PNtIte}
\begin{split}
||\partial_k P^N_tf|| \leq \frac{t^{n_k}}{n_k!} e^{(1+\eta)t} |||f|||_1
\end{split}
\end{equation}
For any $A>0$, choosing $B \geq 1$ such
that
$$2-logB+log(1+\eta)+\frac{1+\eta}{B} \leq -2A,$$
as $n>Bt$, one has
\begin{equation*} 
\begin{split}
& \ \ \frac{t^{n} (1+\eta)^{n}}{n!} e^{(1+\eta)t}
\leq \exp\{n \log \frac{1+\eta}{B}+2n+(1+\eta)\frac{n}{B}\} \\
& \leq
\exp\{-2An\} \leq \exp\{-An-At\}.
\end{split}
\end{equation*}
Replacing $n$ by $n_k$, we conclude the proof.
\end{proof}
The following lemma roughly means that
if the initial data is in a ball $B_{R, \rho}$, then
the dynamics will not go far from this ball
in finite time. Note that the following $P^N_t f_k(x)$
equals to $\E_x[|X^N_k(t)|]$.
\begin{lem}
Let $f_k(y)=|y_k|$ with $k \in \Gamma_N$, then for all $x \in B_{R,\rho}$
\begin{equation} \label{e:1MomPNt}
P^N_t f_k(x) \leq C (1+|k|)^{\rho} e^{\rho d(1+\eta)t}.  \ \ \
\end{equation}
where $C=C(\rho,R, \eta, d, K)$.
\end{lem}
\begin{proof}
For the notional simplicity, we shall drop the index $N$ of the quantities if no confusions arise.
By \eqref{e:DuHam}, we have
\begin{equation*}
P_tf_k(x)=S_t f_k(x)+\int_0^t S_{t-s} [\ \sum_{i \in \Gamma_N}
(\sum_{i \in \Gamma_N \setminus i} a_{ij} y_j+U_{i}(y)) \partial_i P_s
f_k \ ](x) ds.
\end{equation*}
By \eqref{e:DefSt}, \eqref{e:PolEst}, and that of $B_{R, \rho}$,
one has
\begin{equation} \label{e:Stfk}
\begin{split}
|S_t f_k(x)| & \leq S_t[|y_k-e^{-t} x_k|+e^{-t} |x_k|] \\
&=\int_{\R} p \left(\frac{1-e^{-\alpha t}}{\alpha},0,y_k \right) |y_k| dy_k+e^{-t}|x_k| \\
& \leq C+R(1+|k|)^\rho.
\end{split}
\end{equation}
By \eqref{e:TriBar} and the easy fact $|||f_k|||_1=1$, we have
\begin{equation*}
\left|\int_0^t S_{t-s} \left[\sum_{i \in \Gamma_N} U_{i}(y)
\partial_i P_s f_k\right]ds\right| \leq \sup_i ||U_{i}||\int_0^t |||P^N_s f_k|||_1 ds \\
\leq C e^{(\eta+1)t}.
\end{equation*}
Moreover, by the same argument as in \eqref{e:Stfk},
\begin{equation} \label{e:IntSt-sEst}
\begin{split} & \ \ \left|\int_0^t S_{t-s} \left[\sum_{i \in \Gamma_N}
\sum_{j \in \Gamma_N
\setminus i} a_{ij} y_j \partial_i P_s f_k \right](x) ds \right| \\
& \leq \int_0^t  \sum_{i \in \Gamma_N} ||\partial_i P_s f_k|| \sum_{j \in \Gamma_N \setminus
i}a_{ij} S_{t-s} \left[|y_j-e^{-(t-s)}x_j|+e^{-(t-s)}|x_j| \right](x)
ds\\
& \leq C \int_0^t \sum_{i \in \Gamma_N} \sum_{j \in \Gamma_N \setminus
i}a_{ij}||\p_i P_s f_k|| ds+\int_0^t \sum_{i \in \Gamma_N} \sum_{j \in \Gamma_N \setminus i}a_{ij}
||\partial_i P^N_s f_k|| |x_j| ds.\\
\end{split}
\end{equation}
For the first term in the last line, by Assumption \ref{a:AU} (in particular, $a_{ij} \leq 1+\eta$ and
$a_{ij}=0$ if $|i-j|>K$) and \eqref{e:TriBar},
one can easily have
\begin{equation} \label{e:StaParFirMom}
\sum_{i \in \Gamma_N} \sum_{j \in \Gamma_N \setminus
i}a_{ij}||\p_i P_s f_k|| \leq K^d (1+\eta) |||P_s f_k|||_1 \leq K^d(1+\eta) e^{(1+\eta)s}
\end{equation}
As for the second term, let us first estimate the double summations therein '$\sum_{i \in \Gamma_N} \sum_{j \in \Gamma_N \setminus i} \cdots$', the idea is to split the first sum '$\sum_{i \in \Gamma_N}$' into two
pieces '$\sum_{i:|i-k|>Bs}$', '$\sum_{i:|i-k|\leq Bs}$' and control them
 by \eqref{e:expfinite} and \eqref{e:TriBar}. More
 precisely, for any $A>0$, let $B>1$ be chosen as in Lemma \ref{l:FinSpePro},
 by Assumption \ref{a:AU} (in particular, $a_{ij}=0$ for $|i-j|>K$), and the fact
 $|x_j| \leq R(1+|j|)^\rho$,
 the piece '$\sum_{i:|i-k|>Bs}$'can be estimated by
\begin{equation*}
\begin{split}
\sum_{i:|i-k|>Bs} ||\partial_i P_s f_k|| & \sum_{j \in \Gamma_N \setminus
i} a_{ij} |x_j| \leq \sum_{i \in \Gamma_N} e^{-A|i-k|/K-As} R \left(1+|i|+K^d\right)^\rho  \\
& \leq C(K,d,\rho,R, \eta)  \sum_{i \in \Gamma_N} e^{-A|i-k|/K-As}[(1+|k|)^{\rho}+|k-i|^{\rho}] \\
 & \leq C(K,d,\rho,R, \eta,A)(1+|k|)^\rho.
\end{split}
\end{equation*}
As for the other piece, by Assumption \ref{a:AU} again, we have
\begin{equation*}
\begin{split}
 \sum_{i:|i-k| \leq Bs} ||\partial_i P_s f_k|| \sum_{j \in \Gamma_N \setminus
i}& a_{ij} |x_j| \leq \sum_{i:|i-k| \leq Bs}||\p_i P_s f_k|| K^d (1+\eta)R(1+|i|+K^d)^{\rho} \\
& \leq |||P_s f_k|||_1 K^d (1+\eta) R\left(1+|k|+(Bs)^d+K^d\right)^{\rho}   \\
 & \leq C(R, K,d,\eta,B,\rho) s^{\rho d} e^{(1+\eta)s} (1+|k|)^\rho
\end{split}
\end{equation*}
Combining the above two inequalities, one immediately has
$$\sum \limits_{i \in \Gamma_N} \sum_{j \in \Gamma_N \setminus i} a_{ij} ||\partial_i P^N_s f_k||
|x_j| \leq C(R,d,K,\rho,\eta) (1+s)^{\rho d}e^{(1+\eta)s}(1+|k|)^\rho.$$
Hence, plugging the above estimates and \eqref{e:StaParFirMom} into \eqref{e:IntSt-sEst}, we have
\begin{equation*}
\left|\int_0^t S_{t-s} \left[\sum_{i \in \Gamma_N}
\sum_{j \in \Gamma_N
\setminus i} a_{ij} y_j \partial_i P_s f_k \right](x) ds \right| \leq C(d,\eta, R,\rho,K) e^{\rho d (1+\eta)t} (1+|k|)^{\rho}
\end{equation*}
\end{proof}
\subsection{Proof of Theorem \ref{t:ConDyn}}
\begin{proof}
We shall firstly prove that $\lim \limits_{N \rightarrow \infty} P^N_t f(x)$ exists for
any $f \in \mcl D^\infty$, $t>0$ and $x \in \mathbb B$. It suffices to show that $\{P^N_tf(x)\}_N$ is a
Cauchy sequence pointwisely in one $B_{R, \rho}$. \\

Take $x \in B_{R,\rho}$, for any $\Gamma_M \supset \Gamma_N
\supset \Lambda(f)$ with $M>N$, we have
\begin{equation} \label{e:PMN}
\begin{split}
|P^M_tf(x)& -P^N_tf(x)|\leq \left|\int_{0}^{t}
\frac{d}{ds}P^M_{t-s}\left[P^N_sf\right](x) ds \right| \\
& \leq \left|\int_{0}^{t}P^M_{t-s}
\left[(\mathcal{L}_M-\mathcal{L}_N)P^N_sf\right](x) ds \right| \\
& \leq \left|\int_{0}^{t}P^M_{t-s} [\ \sum_{i \in \Gamma_M \setminus
\Gamma_N} (\sum_{j \in \Gamma_M} a_{ij} y_j+U^M_{i}(y))
\partial_i P^N_sf \ ](x) ds \right| \\
& \ \ +\left|\int_{0}^{t}P^M_{t-s} [\ \sum_{i \in \Gamma_N}\left(U^M_{i}(y)-U^N_i(y)\right)
\partial_i P^N_sf\ ](x) ds \right|.
\end{split}
\end{equation}
By \eqref{e:1MomPNt}, (3) of Assumption \ref{a:AU} and \eqref{e:expfinite},
\begin{equation} \label{e:LinCon}
\begin{split}
& \ \ P^M_{t-s}[\sum_{i \in \Gamma_M \setminus \Gamma_N}\sum_{j \in \Gamma_M \setminus i}
a_{ij} |y_j| \cdot ||\partial_i P^N_sf||](x) \\
& \leq \sum_{i \in \Gamma_M \setminus \Gamma_N} ||\partial_i P^N_sf||\sum_{j \in \Gamma_M} a_{ij} P^N_{t-s}[|y_j|](x) \\
& \leq C(t-s,R,d,\rho,\eta)\sum_{i \in \Gamma_M \setminus \Gamma_N} K^d (1+\eta)
(1+|i|+K^{d})^{\rho}||\partial_i P^N_sf|| \\
& \rightarrow 0 \ \ \ (M, N \rightarrow \infty).
\end{split}
\end{equation}
Similarly, by \eqref{e:expfinite} again, as $M, N \rightarrow \infty$,
\begin{equation*} \label{e:BouCon}
\sum_{i \in \Gamma_M \setminus \Gamma_N} |P^M_{t-s}[\ U^M_{i}(y)
\partial_i P^N_sf\ ](x)| \leq \sup_i ||U_{i}||\sum_{i \in \Gamma_M \setminus \Gamma_N}  ||\p_i P^N_s f||
\rightarrow 0.
\end{equation*}
By (3) of Assumption
\ref{a:AU}, the definition of
$U^N_i$ and \eqref{e:expfinite}, the term $'|\cdots|'$ in the last line of \eqref{e:PMN}
can be bounded by
 \begin{equation}
\begin{split}
|\cdots|&=
\left|\int_{0}^{t}P^M_{t-s} [\
\sum_{i \in \p_K(\Gamma_N)} \left(U^M_i(y)-U^N_i(y) \right)
\partial_i P^N_sf\ ](x) ds \right| \\
& \leq  2 \sup_{i} ||U_i|| \int_{0}^{t}
\sum_{i \in \p_K(\Gamma_N)}
||\p_i P^N_sf|| ds  \\
& \rightarrow 0 \ \ \ \ (N \rightarrow \infty)
\end{split}
\end{equation}
where $\p_K(\Gamma_N)=\{i \in \Gamma_N; dist(i,\p \Gamma_N) \leq K\}$
and $\p \Gamma_N$ is the boundary of $\Gamma_N$. \\

Injecting all the above inequalities into \eqref{e:PMN},
we conclude that $\{P^N_t f(x)\}_N$ is a Cauchy sequence for any $t>0$, $x \in B_{\rho,R}$
and $f \in D^{\infty}$. \\

For any $f \in \D^\infty$ and $x \in \mathbb B$, denote
\begin{equation} \label{e:DefPtf}
P_t f(x):=\lim_{N \rightarrow \infty} P^N_t f(x).
\end{equation}
Since $\mcl D^\infty$ is dense in $\mcl B_b(\mathbb B, \R)$ (under the product topology), we can extend
the domain of $P_t$ from $\mcl D^\infty$ to $\mcl B_b(\mathbb B, \R)$.
Thanks to \eqref{e:TriBar} and similar arguments as above, one can pass to
the limit on the both sides of $P^N_{t_1+t_2} f(x)=P^{N}_{t_1} P^N_{t_2}f(x)$
and obtain the semigroup property of $P_t$, i.e. $P_{t_1+t_2} f(x)=P_{t_1}P_{t_2} f(x)$.
It is easy to see that $P_t ({\bf 1})(x)=1$ for all $x \in \mathbb B$ and $P_t f(x) \geq 0$.
Hence $P_t$ is a Markov semigroup on $\mcl B_b(\mathbb B,\R)$.
\end{proof}

\section{Ergodic Theorem \ref{t:ErgThm}}  \label{s:ErgThm}
\subsection{Auxiliary Lemmas} To prove the ergodicity result,
we need the following three auxiliary lemmas.
\\

Let $S_t$ be the product semigroup defined in \eqref{e:DefSt}, for any $f \in \D^\infty$,
by replacing $\Lambda_0$ in \eqref{e:AsySt} by $\Lambda(f)$, one can easily have
\begin{equation*}
\begin{split}
& \ \ |S_{t_2}f(0)-S_{t_1}f(0)| \\
& \leq |\Lambda(f)| \cdot ||f||
\int_{\R}\left|p \left(\frac{1-e^{-\alpha t_2}}{\alpha},0,y \right)-p
\left(\frac{1-e^{-\alpha t_1}}{\alpha},0,y\right)\right| dy \\
& \rightarrow 0 \ \ \ (t_1,t_2 \rightarrow \infty).
\end{split}
\end{equation*}
 However, the above convergence speed depends on $|\Lambda(f)|$, the size of $f$. This type of convergence
 is not enough to control some
limit in the interacting system. Alternatively, we shall use the information of $f \in \mathcal{D}^1$,
i.e.
$|||f|||_1<\infty$,
and prove
the convergence speed \emph{asymptotically} depends on $\Lambda(f)$, this will make
some room to uniformly control the second term on the r.h.s. of \eqref{e:St2sSt1s}.
More precisely, we have
\begin{lem} \label{l:ConStf}
Let $S_t$ be the product semigroup generated by \eqref{e:OrnUhlPro}.
Given any $f \in \mathcal{D}^\infty$ and $\Lambda_0 \subset \Lambda(f)$,
for any $t_2,t_1 \geq 0$, we have
\begin{equation*} \label{e:DelErgS}
\begin{split}
|S_{t_2}f(0)-S_{t_1}f(0)| & \leq |\Lambda_0| \cdot ||f|| \int_{\R} |p \left(\frac{1-e^{-\alpha t_2}}{\alpha},0,y \right)-p \left(\frac{1-e^{-\alpha t_1}}{\alpha},0,y\right)|dy\\
& \ \ +C
\sum_{i \in \Lambda(f) \setminus \Lambda_0} ||\p_if||.
\end{split}
\end{equation*}
where $C>0$ is some constant only depending on the parameter $\alpha$.
\end{lem}
\begin{proof}
We can easily have
\begin{equation} \label{e:AsySt}
\begin{split}
& \ \ |S_{t_2}f(0)-S_{t_1}f(0)| \\
& \leq \left|\int_{{\R}^{\Lambda_0}} \left[\prod_{i \in \Lambda_0} p \left(\frac{1-e^{-\alpha t_2}}{\alpha},0,y_i \right)-\prod_{i \in \Lambda_0} p \left(\frac{1-e^{-\alpha t_1}}{\alpha},0,y_i \right)\right] f(y^{\Lambda_0},0) dy \right| \\
&\ \ +\int_{{\R}^{\Lambda(f)}} \prod_{i \in \Lambda(f)}
p \left(\frac{1-e^{-\alpha t_1}}{\alpha},0,y_i \right) |f(y)-f(y^{\Lambda_0},0)|dy \\
&\ \ +\int_{{\R}^{\Lambda(f)}} \prod_{i \in \Lambda(f)}
p \left(\frac{1-e^{-\alpha t_2}}{\alpha},0,y_i \right) |f(y)-f(y^{\Lambda_0},0)|dy \\
&=I_0+I_1+I_2
\end{split}
\end{equation}
It is easy to see
\begin{equation*}
 I_0 \leq |\Lambda_0| \cdot ||f|| \int_{\R}
 \left|p \left(\frac{1-e^{-\alpha t_2}}{\alpha},0,y \right)-p
 \left(\frac{1-e^{-\alpha t_1}}{\alpha},0,y\right)\right|
  dy.
\end{equation*}
$I_1$ and $I_2$ can be bounded in the same way as the following: By \eqref{e:PolEst},
\begin{equation*}
\begin{split}
I_1 \leq \sum_{i \in \Lambda(f) \setminus \Lambda_0} ||\p_i f|| \int_{{\R}}
p \left(\frac{1-e^{-\alpha t_1}}{\alpha},0,y_i \right) |y_i| dy_i
\leq C \sum_{i \in \Lambda(f) \setminus \Lambda_0} ||\partial_i f||.
\end{split}
\end{equation*}
\end{proof}
The next lemma claims that if $\eta$ is sufficiently small, then the gradient
of $P^N_t f$ uniformly decays in an exponential speed.
\begin{lem} \label{l:GradientDecayL2}
There exists some $c>0$ such that if $\eta<c$, then, for any $m \in
{\mathbb N}$, we have
\begin{equation} \label{e:GB}
|\nabla P^N_t f|^{2m} \leq e^{-2m\beta t} P^N_t|\nabla f|^{2m} \ \ \
\forall \ f \in \D^\infty
\end{equation}
where $\beta=\beta(a,U)>0$. In particular,
\begin{equation} \label{e:GB2}
||\p_i P^N_t f|| \leq e^{-\beta t} |||f|||_1.
\end{equation}
\end{lem}
\begin{proof}
For the notational simplicity, we drop the index of the quantities if no confusions arise.
For any $f \in \D^\infty$ and any $m \in {\mathbb N}$, we have the
following calculation
\begin{equation*}
\begin{split}
\frac{d}{ds}P_s |\nabla P_{t-s}f|^{2m}&=P_s[\mathcal{L}_N|\nabla
P_{t-s} f|^{2m}-2m |\nabla P_{t-s}f|^{2(m-1)} \nabla P_{t-s} f
\cdot \mathcal{L}_N \nabla P_{t-s}
f] \\
&\ \ \ +2m P_s(|\nabla P_{t-s}|^{2(m-1)} \nabla P_{t-s} f \cdot
[\mathcal{L}_N,\nabla] P_{t-s} f) \\
& \geq 2m P_s(|\nabla P_{t-s}|^{2(m-1)} \nabla P_{t-s} f \cdot
[\mathcal{L}_N,\nabla] P_{t-s} f),
\end{split}
\end{equation*}
since we have $P_tF^{2m}\geq (P_tF)^{2m}$ and
therefore $\lim \limits_{t \rightarrow 0+} \frac{P_tF^{2m}-F^{2m}}{t}
\geq \lim \limits_{t \rightarrow 0+}
\frac{(P_tF)^{2m}-F^{2m}}{t}$, i.e.
$\mathcal{L}_N F^{2m}-2m F^{2(m-1)} F \mathcal{L}_N F \geq 0.$ \\

Hence,
\begin{equation*}
\begin{split}
& \ \ \frac{d}{ds}P_s |\nabla P_{t-s} f|^{2m} \\
& \geq -2m P_s \left \{|\nabla P_{t-s} f|^{2(m-1)} \sum_{i \in \Gamma_N}
\sum_{j \in \Gamma_N \setminus i} (a_{ij}+\p_j U_{i}) (\p_i
P_{t-s} f) (\p_j P_{t-s} f) \right\} \\
& \ \ \ +2m P_s \left \{|\nabla P_{t-s} f|^{2(m-1)}\sum_{i \in \Gamma_N} (1-\p_iU_i)|\p_i P_{t-s}f|^2 \right\}. \\
\end{split}
\end{equation*}
Define the quadratic form
$$Q(\xi, \xi):=-\sum_{i \in \Gamma_N}
\sum_{j \in \Gamma_N \setminus i} (a_{ij}+\p_j U_{i})\xi_i \xi_j+\sum_{i \in \Gamma_N}
(1-\p_iU_{i}) \xi^2_i$$
for any $\xi \in \R^{\Gamma_N}$. As $\eta$ in (4) of Assumption \ref{a:AU} is small enough,
it is strictly positive
definite, that is, we have some constant $c, \beta>0$ such that as $\eta<c$,
$$Q(\xi,\xi) \geq \beta |\xi|^2.$$
Therefore,
\begin{equation*}
\frac{d}{ds}P_s |\nabla P_{t-s} f|^{2m}  \geq 2 m \beta P_s\left
(|\nabla P_{t-s} f|^{2m}\right),
\end{equation*}
from which we immediately obtain \eqref{e:GB}.
\eqref{e:GB2} is an immediate corollary of \eqref{e:GB}.
\end{proof}
The last lemma can be taken as the second order finite speed propagation of interactions,
which is also due to the finite range interactions.
\begin{lem} \label{l:TriBar2}
There exists some constant $C=C(a,U)>0$ so that
\begin{equation} \label{e:2BarNorEst}
|||P^N_tf|||_2 \leq e^{Ct}(|||f|||_1+|||f|||_2).
\end{equation}
Moreover, there exists some constant $B=B(a,U) \geq 1$ so that as
$$dist(j, \Lambda(f)), dist(j, \Lambda(f)) \geq Bt,$$ we have
\begin{equation} \label{e:TriBar2}
||\partial_{jk}P^N_tf|| \leq h_{jk} (|||f|||_1+|||f|||_2).
\end{equation}
where $\{h_{jk}\}_{j,k \in \Z^d}$ is a constant sequence only depending on $a$ and $U$, and
 $$h_{jk} \geq 0,
\sum \limits_{j,k \in \Z^d} h_{jk}<\infty.$$
\end{lem}
\begin{proof}
Furthermore, by differentiating $P^N_{t-s} \p_{kj} P^N_s f$ (on $s$), we have
\begin{equation} \label{e:PNtIte2}
\begin{split}
||\partial_{kj} P^N_tf|| &\leq ||\p_{kj} f||+\int_0^t ||[\partial_{kj}, \mathcal{L}_N] P^N_s f||ds \\
& \leq ||\p_{kj} f||+\int_0^t \sum \limits_{i \in
\Gamma_N} [(|a_{ik}|+||\partial_k U_i||) ||\partial_{ji}P_s^Nf||+(|a_{ij}|+||\partial_j U_i||)||\partial_{ki}P_s^Nf||]ds \\
& \ \ + \int_0^t \sum \limits_{i \in \Gamma_N} ||\partial_{jk} U_i|| ||\partial_{i}P_s^Nf|| ds
\end{split}
\end{equation}
By some iteration similar to that for \eqref{e:PNtIte}, we have some convergent sequence $\{h_{jk}\}_{j,k \in {\Z}^d}$
only depending on $a$ and $U$ so that ${||\partial_{kj} P^N_tf||}_{j,k \in {\Z}^d} \leq h_{jk}
(|||f|||_1+|||f|||_2)$ (The
$h_{jk}$ here plays a similar role as that of $e^{-At-An_k}$ in \eqref{e:expfinite}).
Moreover, Summering $i,k$ over $\Z^d$ in \eqref{e:PNtIte2},
we immediately have \eqref{e:2BarNorEst}.
\end{proof}
\subsection{Proof of Theorem}
The proof of Theorem is lengthy, we first prove the following Proposition
\ref{p:ErgAt0}, which is the crucial step.
\begin{prop} \label{p:ErgAt0}
For any $f \in \mcl D^\infty$, the limit $\lim_{t \rightarrow \infty}P_tf(0)$ exists.
\end{prop}
\begin{proof}
To prove the proposition, it suffices to show that for arbitrary $\e>0$,
there exists some constant $T>0$ such that
as $t_2 \geq t_1 \geq T$
\begin{equation} \label{e:Pt2-Pt1e}
|P_{t_2} f(0)-P_{t_1}f(0)| \leq 4 \e.
\end{equation}
Let us show this inequality in the following three steps:
 \\

 \noindent \emph{\underline{Step 1: Proof of \eqref{e:Pt2-Pt1e}}}.
 For any $t_2 \geq t_1 \geq 0$, by triangle inequality, we have
\begin{equation*} \label{e:01}
\begin{split}
|P_{t_2}f(0)-P_{t_1}f(0)| & \leq |P_{t_2}f(0)-P^N_{t_2}f(0)|\\
&\ \ +|P^N_{t_2}f(0)-P^N_{t_1}f(0)|+|P_{t_1}f(0)-P^N_{t_1}f(0)|.
\end{split}
\end{equation*}
By \eqref{e:DefPtf}, there
exists some $N_0=N_0(t_1,t_2) \in {\mathbb N}$ such that as $N>N_0$
\begin{equation} \label{e:PtPNt}
|P_{t_2}f(0)-P^N_{t_2}f(0)|+|P_{t_1}f(0)-P^N_{t_1}f(0)|<\varepsilon.
\end{equation}
Hence, to conclude the proof, we only need to show that, there exists some constant $T>0$,
which is
independent of $N$, such that as $t_2, t_1>T$
\begin{equation} \label{e:PNt2-PNt1}
|P^N_{t_2} f(0)-P^N_{t_1} f(0)| <3\e.
\end{equation}
\ \\
\indent By \eqref{e:DuHam}, the l.h.s. of \eqref{e:PNt2-PNt1}
can be split into the following \emph{two} pieces:
\begin{equation} \label{e:ErgodicAt0}
\begin{split}
|P^N_{t_2}f(0)&-P^N_{t_1}f(0)| \leq
|S_{t_2}f(0)-S_{t_1}f(0)| \\
&+|\int_{0}^{t_2} S_{t_2-s} \left[\sum \limits_{i
\in \Gamma_N}(\sum_{j \in \Gamma_N \setminus i} a_{ij} x_j+
U_{i,N}) \partial_i P^N_s f\right](0)ds\\
&\ \ \ \ -\int_{0}^{t_1} S_{t_1-s} \left[\sum \limits_{i \in
\Gamma_N}(\sum_{j \in \Gamma_N \setminus i} a_{ij} x_j+ U_{i,N})
\partial_i P^N_s f\right](0)ds|.
\end{split}
\end{equation}
By the fact $f \in
\mathcal{D}^{\infty}$ and \eqref{e:DefSt}, we have some constant
$T_0=T_0(|\Lambda(f)|,||f||)>0$ such that as $t_1, t_2 \geq T_0$
\begin{equation} \label{e:ErgSt2St1}
|S_{t_2}f(0)-S_{t_1}f(0)| \leq \e.
\end{equation} \\ \
\indent From the above inequality,
to conclude the proof of \eqref{e:PNt2-PNt1}, it suffices to bound the second term '$|\cdots|$' on
the r.h.s. of \eqref{e:ErgodicAt0} by $2\e$.
To this end, we split it into three pieces as the following
\begin{equation*}
\begin{split}
& J_1=\int_{L}^{t_2} S_{t_2-s} \left[\sum \limits_{i \in
\Gamma_N}(\sum_{j \in \Gamma_N \setminus i} a_{ij} x_j+ U^N_{i})
\partial_i P^N_s f\right](0)ds  \\
& J_2=\int_{L}^{t_1} S_{t_1-s} \left[\sum \limits_{i
\in \Gamma_N}(\sum_{j \in \Gamma_N \setminus i} a_{ij} x_j+
U^N_{i}) \partial_i P^N_s f \right](0)ds \\
&J_3=\int_{0}^{L} (S_{t_2-s}-S_{t_1-s}) \left[\sum \limits_{i \in
\Gamma_N} (\sum_{j \in \Gamma_N \setminus i} a_{ij} x_j+U^N_{i})
\partial_i P^N_s f \right](0)ds
\end{split}
\end{equation*}
where $0<L<t_1$ is some large number to be determined later, and show that there exists some
constant $T>0$, which is independent of $N$ and larger than $L$, such that as $t_1,t_2 \geq T$
\begin{equation} \label{e:J1J2J3e}
|J_1|+|J_2|+|J_3| \leq 2\e.
\end{equation}
\\
\noindent \emph{\underline{Step 2: Proof of \eqref{e:J1J2J3e}}}.
As for $|J_1|$, choose some $m \in \N$ with  $2m/(2m-1)<\alpha$, by H$\ddot o$der's
inequality, \eqref{e:PolEst}, (3) of Assumption \ref{a:AU} and \eqref{e:GB},
\begin{equation*}
\begin{split}
& \ \ \left|\int_{L}^{t_2} S_{t_2-s} \left[\sum \limits_{i \in \Gamma_N}
\sum_{j \in \Gamma_N \setminus i} a_{ij} x_j
\partial_i P^N_s f\right](0)ds\right| \\
 & \leq \left|\int_{L}^{t_2}  \sum \limits_{i \in \Gamma_N}
\sum_{j \in \Gamma_N \setminus i} a_{ij} \{S_{t_2-s}
[|x_j|^{\frac{2m}{2m-1}}](0)\}^{\frac{2m-1}{2m}} \{S_{t_2-s}
[|\partial_i P^N_s f|^{{2m}}](0)\}^{\frac{1}{2m}} ds \right| \\
& {\leq} C(m) \eta K^d \left|\int_{L}^{t_2}
\{S_{t_2-s} [|\nabla P^N_s f|^{{2m}}](0)\}^{\frac{1}{2m}} ds\right|\\
& {\leq} C(m, \eta, K, d)
|||f||| \ (e^{-L
\beta}-e^{-t_2 \beta}).
\end{split}
\end{equation*}
By a similar arguments, we have
\begin{equation*}
\begin{split}
\left|\int_{L}^{t_2} S_{t_2-s} \left[\sum \limits_{i \in \Gamma_N} U^N_{i}
\partial_i P^N_s f \right](0) ds\right| &\leq  \sup_i ||U_{i}|| \int_{L}^{t_2} \left \{S_{t_2-s}
[|\nabla P^N_s f|^2](0) \right \}^{1/2} ds \\
& {\leq} C|||f||| \ (e^{-\beta L}-e^{-\beta t_2}).
\end{split}
\end{equation*}
Hence,
\begin{equation*} \label{e:J1}
|J_1| \leq C(m, \eta, K,d) |||f||| \ (e^{-L
\beta}-e^{-t_2 \beta}).
\end{equation*}
By the same method as for $J_1$, one has
\begin{equation*} \label{e:J2}
|J_2| \leq C(m,\eta,K,d) |||f||| \ (e^{-L
\beta}-e^{-t_1 \beta}).
\end{equation*}
Therefore, there exists some constant $L>0$, which is independent of
$N$, such that as $t_1,t_2 \geq L$,
\begin{equation} \label{e:J1J2e}
|J_1|+|J_2| \leq \e
\end{equation}
\ \\
\indent To bound $|J_3|$, choose some cube $\Lambda \subset \Z^d$, which is centered at
 $0$, such that $\Lambda(f) \subset \Lambda \subset \Gamma_N$ and split $|J_3|$ into three pieces:
\begin{equation*} \label{e:J3}
\begin{split}
|J_3| & \leq \left|\int_{0}^{L} S_{t_2-s} \left[\sum \limits_{i \in \Gamma_N \setminus \Lambda} (\sum_{j \in \Gamma_N \setminus i} a_{ij} x_j+U^N_{i})
\partial_i P^N_s f \right](0)ds\right|\\
& \ +\left|\int_{0}^{L} S_{t_2-s} \left[\sum \limits_{i \in \Gamma_N \setminus \Lambda} (\sum_{j \in \Gamma_N \setminus i} a_{ij} x_j+U^N_{i})
\partial_i P^N_s f \right](0)ds\right|\\
& \ +\left|\int_{0}^{L} (S_{t_2-s}-S_{t_1-s}) \left[\sum \limits_{i \in \Lambda} (\sum_{j \in \Gamma_N \setminus i} a_{ij} x_j+U^N_{i})
\partial_i P^N_s f\right](0)ds\right| \\
&=I_1+I_2+I_3.
\end{split}
\end{equation*}
where $I_1,I_2$ and $I_3$ denotes the three terms on the r.h.s. of the above inequality in order. \\

By Lemma \ref{l:FinSpePro} and \eqref{e:PolEst}, for any $A>0$, choose $B \geq 1$ as in Lemma
\ref{l:FinSpePro}, as $\Lambda$ is sufficiently large, which is chosen independent of
$N$, so that $dist(\Lambda^c, \Lambda(f)) \geq BL$ and the following
'$\sum_{i \in \Z^d \setminus \Lambda} e^{-n_i A}$' is sufficiently small,
by (3) of Assumption \ref{a:AU} and \eqref{e:PolEst}, we have
\begin{equation}
\begin{split}
& \ \ \left|S_{t_2-s} \left[\sum \limits_{i \in \Gamma_N \setminus \Lambda}
\sum_{j \in \Gamma_N \setminus i} a_{ij} x_j
\partial_i P^N_s f\right](0)\right| \\
& \leq \sum \limits_{i \in \Gamma_N \setminus \Lambda} e^{-n_i A-As}|||f|||_1
\sum_{j \in \Gamma_N \setminus i}a_{ij} S_{t_2-s}[|x_j|](0) \\
& \leq C K^d (1+\eta) \sum \limits_{i \in \Gamma_N \setminus \Lambda} e^{-n_i A}
|||f|||_1 \leq \frac{\e}{8L}
\end{split}
\end{equation}
where $n_i=[dist(i, \Lambda(f))/K]$. Similarly, choose some sufficiently large $\Lambda$ independent of
$N$, one has
$$\left|S_{t_2-s} \left[\sum \limits_{i \in \Gamma_N \setminus \Lambda}
U^N_i(y)
\partial_i P^N_s f\right](0)\right| \leq \frac{\e}{8L}.$$
Therefore, $I_1 \leq \frac{\e}{4L}$, one has $I_2 \leq \frac{\e}{4L}$ by the same method. \\

Hence, as $\Lambda$ is sufficiently large, which is independent of $N$, one has
 \begin{equation} \label{e:I1I2e}
I_1+I_2 \leq \frac{\e}{2L}.
\end{equation}
In the following Step 3, we shall prove that there exists some constant $T>L$,
which is \emph{independent} of $N$ but \emph{depends} on $|\Lambda|,L$, so that
as $t_1,t_2 \geq T$
\begin{equation} \label{e:I3e}
I_3 \leq \frac \e{2L},
\end{equation}
thus $|J_3| \leq \e$. Combining this with \eqref{e:J1J2e}, we conclude the proof of \eqref{e:J1J2J3e}.
\\

\noindent \emph{\underline{Step 3: Proof of \eqref{e:I3e}}}. Let us first consider the linear part of
$I_3$, it can be split into two pieces:
\begin{equation} \label{e:J3Aux}
\begin{split}
& \ \ \left|(S_{t_2-s}-S_{t_1-s}) [\ \sum \limits_{i \in \Lambda} \sum_{j \in \Gamma_N \setminus i} a_{ij} x_j
\partial_i P^N_s f\ ](0)\right| \\
& \leq  \left|(S_{t_2-s}-S_{t_1-s}) [\ \sum \limits_{i \in \Lambda} \sum_{j \in \Gamma_N \setminus 2 \Lambda} a_{ij} x_j
\partial_i P^N_s f\ ](0)\right| \\
&\ \ +\left|(S_{t_2-s}-S_{t_1-s}) [\ \sum \limits_{i \in \Lambda} \sum_{j \in 2 \Lambda \setminus i} a_{ij} x_j \partial_i P^N_s f\ ](0)\right|,
\end{split}
\end{equation}
where $2\Lambda$ is the cube centered at $0$, whose edge length is two times as that of $\Lambda$.
By (3) of Assumption \ref{a:AU}, if $\Lambda$ is large enough,
the first term on the r.h.s. of \eqref{e:J3Aux} are zero since $a_{ij}=0$ for all $i,j$ therein. \\

As for the last term on the r.h.s. of \eqref{e:J3Aux}, we need
some delicate analysis as the following: Take some large number $H>0$
(to be determined later) and split the term into two pieces, and control
them by some uniform integrability and Lemma \ref{l:ConStf}, more precisely, we have
\begin{equation} \label{e:St2sSt1s}
\begin{split}
& \ \ \left|(S_{t_2-s}-S_{t_1-s}) [\ \sum \limits_{i \in \Lambda} \sum_{j \in 2 \Lambda \setminus i}
a_{ij} x_j \partial_i P^N_s f\ ](0)\right| \\
& \leq  \left|(S_{t_2-s}-S_{t_1-s}) [\ \sum \limits_{i \in \Lambda} \sum_{j \in 2 \Lambda \setminus i}
a_{ij} x_j  \left(1-\chi(\frac{|x_j|}{H})\right) \partial_i P^N_s f\ ](0)\right| \\
& \ + \left|(S_{t_2-s}-S_{t_1-s}) [\ \sum \limits_{i \in \Lambda} \sum_{j \in 2 \Lambda \setminus i}
a_{ij} x_j \chi(\frac{|x_j|}{H}) \partial_i P^N_s f\ ](0)\right|
\end{split}
\end{equation}
where $\chi:[0, \infty) \rightarrow [0, \infty)$ is some smooth function such that
$\chi(x)=1$ as $0 \leq x \leq 1$ and $\chi(x)=0$ as $x \geq 2$.
By \eqref{e:PolEst} with $1<\beta<\alpha$ therein, the fact $a_{ij} \leq 1+\eta$ and
\eqref{e:TriBar},
(note that $s \leq L$),
the first term on the r.h.s. of \eqref{e:St2sSt1s}
can be bounded by
\begin{equation*}
\begin{split}
& \ \ \ 2^d|\Lambda| (1+\eta) |||P^N_s f|||_1 \left(S^0_{t_2-s}+S^0_{t_1-s}\right)
\left[|x| \left(1-\chi(\frac{|x|}{H})\right)\right] \\
& \leq 2^d|\Lambda| (1+\eta) e^{L(1+\eta)} |||f|||_1 \left(S^0_{t_2-s}+S^0_{t_1-s}\right)
\left[|x| \left(1-\chi(\frac{|x|}{H})\right)\right] \\
& \leq \frac{\e}{12L}
\end{split}
\end{equation*}
as $H>H_0$ if $H_0=H_0(d,|\Lambda|,L,\eta)>0$ is sufficiently large. \\

As for the second term of \eqref{e:St2sSt1s},
by \eqref{e:TriBar}, the fact $s \leq L$ and Lemma \ref{l:ConStf}, one can
choose some cube $\tilde \Lambda \supset 2\Lambda$ (to be determined later), so that
\begin{equation*}
\begin{split}
& \ \ \left|(S_{t_2-s}-S_{t_1-s}) [\ \sum \limits_{i \in \Lambda}
\sum_{j \in 2 \Lambda \setminus i} a_{ij} x_j \chi(\frac{|x_j|}{H}) \partial_i P^N_s f\ ](0)\right| \\
& \leq C(H,L,\eta,|\Lambda|,|\tilde \Lambda|,d) |||f|||_1 \int_{\R}
\left|p \left(\frac{1-e^{-\alpha (t_2-s)}}{\alpha},0,y \right)-p
\left(\frac{1-e^{-\alpha (t_1-s)}}{\alpha},0,y\right)\right|dy \\
&\ \ +C(H,|\Lambda|) \sum_{j \in \Gamma_N \setminus \tilde \Lambda} ||\p_{ij}P^N_s f||.
\end{split}
\end{equation*}
Therefore, by Lemma \ref{l:TriBar2} and the fact $s \leq L$,
choose some sufficiently large $\tilde \Lambda$ (independent of $N$),
the second term on the r.h.s. of the above inequality can be uniformly
bounded by $\frac\e{12L}$.
Moreover, since $s \leq L$, there exists some
constant $T=T(H,L,\eta,|\Lambda|,|\tilde \Lambda|,d)>0$ so that as $t_1,t_2 \geq T$ the
first term is also bounded by
$\frac\e{12L}$. \\

Combining the above estimates with \eqref{e:J3Aux}, we have some constant $T>0$,
which is independent of $N$, so that
\begin{equation}
\left|(S_{t_2-s}-S_{t_1-s}) [\ \sum \limits_{i \in \Lambda} \sum_{j \in \Gamma_N \setminus i} a_{ij} x_j
\partial_i P^N_s f\ ](0)\right| \leq \frac{\e}{4L}
\end{equation}
By the same arguments, one has
\begin{equation}
\left|(S_{t_2-s}-S_{t_1-s}) [\ \sum \limits_{i \in \Lambda} \sum_{j \in \Gamma_N \setminus i} \p_j U_i^N
\partial_i P^N_s f\ ](0)\right| \leq  \frac{\e}{4L}.
\end{equation}
Combining the above two inequalities, we immediately conclude this step.
\end{proof}

\begin{proof} [{\bf Proof of Theorem \ref{t:ErgThm}}]
Denote $\ell (f)=\lim_{t \rightarrow \infty} P_tf(0)$, which is the limit in Proposition
\ref{p:ErgAt0}. We shall split the proof of the theorem into the following two steps.  \\

\noindent \emph{\underline{Step 1: $\lim_{t \rightarrow \infty} P_t f(x)=\ell (f)
\ \ \forall \ x \in \mathbb B$}}.
It suffices to show that the limit
is true on one ball $B_{R, \rho}$. By triangle inequality,
\begin{equation} \label{e:05}
\begin{split}
|P_tf(x)-\ell(f)| &\leq |P_tf(x)-P^N_tf(x)|+|P^N_tf(x)-P^N_tf(0)|
\\
&\ \ \ +|P^N_tf(0)-P_tf(0)|+|P_tf(0)-\ell(f)|
\end{split}
\end{equation}
For arbitrary $\varepsilon>0$, by Proposition \ref{p:ErgAt0}, there exists some $T_0>0$
such that, as $t>T_0$,
\begin{equation} \label{e:limit at 0}
|P_tf(0)-\ell(f)|<\varepsilon.
\end{equation}
By Theorem \ref{t:ConDyn}, there exists some $N(t) \in {\mathbb N}$ such
that as $N>N(t)$
\begin{equation} \label{e:N approximation at point}
|P_tf(x)-P^N_tf(x)|<\varepsilon,\ |P^N_tf(0)-P_tf(0)|<\varepsilon.
\end{equation}
We remain to show that there exists some constant $T_1>0$, which is independent of
$N$, so that as $t \geq T_1$,
\begin{equation} \label{e:PNtfx-PNtf0}
|P^N_tf(x)-P^N_tf(0)| \leq \e.
\end{equation}
\ \\
\indent By fundamental calculus, one has
\begin{equation*} \label{e:expfinite4}
\begin{split}
& \ \ \left|P^N_tf(x)-P^N_tf(0)\right|
=\left|\int_0^1 \frac{d}{d\lambda} P^N_t f(\lambda x) d\lambda \right| \\
& \leq \int_0^1 \sum_{i \in \Gamma_N} |\p_i P^N_t f(\lambda x)| |x_i|  d\lambda
 \leq \sum_{i \in \Gamma_N} ||\p_i P^N_t f|| |x_i|
\end{split}
\end{equation*}
For any constant $A>0$, choose $B \geq 1$ as in Lemma \ref{l:FinSpePro}, and also a
cube $\Lambda(f) \subset \Lambda \subset \Gamma_N$ such that
$dist(\Gamma_N \setminus \Lambda, \Lambda(f))=[BtK]+1$ (up to some $O(1)$ correction),
therefore, by \eqref{e:GB2} and \eqref{e:expfinite},
\begin{equation*} 
\begin{split}
\sum_{i \in \Gamma_N} ||\p_i P^N_t f|| |x_i|
& \leq \sup_{i \in \Lambda} |x_i| |\Lambda| e^{-\beta t} |||f|||_1 \\
& \ \ \ +
\sum_{i \in \Gamma_N \setminus \Lambda} e^{-An_i-At} |||f|||_1 R(1+|i|)^\rho \\
& \leq e^{-\beta t} |||f|||_1 R(|\Lambda(f)|+BKt+1)^{\rho+d}+C(A,R) e^{-At} |||f|||_1
\end{split}
\end{equation*}
where $n_i$ is defined in Lemma \ref{l:FinSpePro}, since
$\sum_{i \in \Gamma_N \setminus \Lambda} e^{-A n_i}R(1+|i|)^\rho<\infty$
uniformly on $N$, thus there exists some constant $T_1=T_1(A,K,R,\rho,\Lambda,f)$ so that
as $t \geq T_1$,
$$\sum_{i \in \Gamma_N} ||\p_i P^N_t f|| |x_i| \leq \e.$$

Take $T=\max\{T_0,T_1\}$, as $t \geq T$, we immediately have
$$|P_t f(x)-\ell (f)| \leq 4 \e.$$
\ \\
\noindent \emph{\underline{Step 2: $P_t$ is strongly mixing.}}
From the definition of $\ell(f)$, it is clear that
$\ell(\cdot)$ is a linear functional on $\D^\infty$. Since $\mathbb B$
is locally compact, by Riesz representation theorem
for linear functional (see a nice introduction in \cite{Tao09}), there exists
some unsigned Randon measure $\mu$ on $\mathbb B$. From the easy fact
$P_t ({\bf 1})=1$, we have $\mu(\mathbb B)=1$, thus $\mu$ is a probability measure.
On the other hand, by the Riesz representation theorem again, for each fixed $t$,
$P_t f(x)$ also admits a probability measure $P^*_t \delta_x$. By Step 1 and
the fact that $\D^\infty$ is dense in $\mcl B_b(\mathbb B,\R)$ (under product topology),
we immediately have $P^{*}_t \delta_x \rightarrow \mu$ in weak sense, which also implies
that the semigroup $P_t$ is strongly mixing.
\end{proof}
\section{Appendix: Formal derivation of \eqref{e:ProRepU}}
Suppose that the Fourier transforms for the solution $u(t)$
and $f$ exist, then the Fourier transform of the equation \eqref{e:kola} is
\begin{equation}
\begin{cases}
\partial_t \hat{u}=-|\lambda|^{\alpha} \hat{u}+\hat{u}+\lambda \partial_{\lambda}\hat{u} \\
\hat{u}(0)=\hat{f}
\end{cases}
\end{equation}
where  '$\hat{\ \ }$'  denotes the Fourier transform of functions. Suppose $\lambda>0$, set
$\nu=\ln \lambda,\ \ \hat{v}=e^{-t}\hat{u}(e^{\nu}), \ \ \hat{g}(\nu)=\hat{f}(e^{\nu})$, we have
\begin{equation}
\begin{cases}
\partial_t \hat{v}=-e^{\alpha \nu} \hat{v}+\partial_{\nu}\hat{v} \\
\hat{v}(0)=\hat{g}(\nu)
\end{cases}
\end{equation}
Suppose $\hat{g}$ is positive, set
$\hat{w}=ln \hat{v},$
the equation for $\hat{w}$ is
\begin{equation}
\begin{cases}
\partial_t \hat{w}=-e^{\alpha \nu}+\partial_{\nu}\hat{w} \\
\hat{v}(0)=ln \hat{g}(\nu)
\end{cases}
\end{equation}
It is easy to solve the above equation by
$\hat{w}(t)=ln \hat{g}(\nu+t)-e^{\alpha \nu} \frac{e^{\alpha t}-1}{\alpha},$
thus
$$\hat{w}(t)=\hat{g}(\nu+t)\exp\{-e^{\alpha \nu} \frac{e^{\alpha t}-1}{\alpha} \}$$
and
\begin{equation*}
\hat{u}(t)=\hat{g}(\nu+t)\exp\{t-e^{\alpha \nu} \frac{e^{\alpha t}-1}{\alpha} \}=\hat{f}(e^t \lambda )\exp\{t-|\lambda|^{\alpha} \frac{e^{\alpha t}-1}{\alpha} \}.
\end{equation*}
Hence, by Parseval's Theorem, we have
\begin{equation*}
\begin{split}
u(t)&=\frac{1}{\sqrt{2 \pi}} \int \limits_{{\bf R}} \hat{f}(e^t \lambda )\exp\{t-|\lambda|^{\alpha} \frac{e^{\alpha t}-1}{\alpha} \}e^{i\lambda x}d\lambda \\
&=\int \limits_{{\bf R}} \hat{f}(\lambda )\frac{1}{\sqrt{2 \pi}}\exp\{-|\lambda|^{\alpha}
\frac{1-e^{-\alpha t}}{\alpha}+i\lambda e^{-t}x\}d\lambda \\
&=\int \limits_{{\bf R}} p\left(\frac{1-e^{-\alpha
t}}{\alpha};e^{-t}x, y\right)f(y)dy
\end{split}
\end{equation*}
\bibliographystyle{amsplain}

\end{document}